\documentclass[12pt]{amsart}
\usepackage{bbm}
\usepackage[top=1in, bottom=1in, left=0.8in, right=0.8in]{geometry}

\usepackage{amsmath,amscd}
\usepackage{amssymb}
\usepackage{amsthm}
\usepackage{comment}
\usepackage{graphicx, xcolor}
\usepackage[abbrev,nobysame,alphabetic]{amsrefs}
\usepackage{mathrsfs}
\usepackage[ocgcolorlinks, linkcolor=blue]{hyperref}

\usepackage{bm}
\usepackage{bbm}
\usepackage{url}

\usepackage[utf8]{inputenc}
\usepackage{mathtools,amssymb}
\usepackage{esint}
\usepackage{tikz}
\usepackage{dsfont}
\usepackage{relsize}
\usepackage{url}
\usepackage{xcolor}
\usepackage{graphicx}
\usepackage{mathrsfs}
\usepackage[shortlabels]{enumitem}
\usepackage{lineno}
\usepackage{amsmath}
\usepackage{enumitem}
\usepackage{amsthm} 
\usepackage{verbatim}
\usepackage{dsfont}
\numberwithin{equation}{section}

\allowdisplaybreaks

\graphicspath{{images/}}

\newtheorem{thm}{Theorem}[section]

\newtheorem{lem}[thm]{Lemma}

\theoremstyle{definition}
\newtheorem{rem}[thm]{Remark}

\newtheorem*{claim*}{Claim}
\theoremstyle{remark}

\numberwithin{equation}{section}

\newcommand{\R}{{\mathbb R}}

\title[]{Eigenfunction equivalence for the fractional Laplace-Beltrami operator and the classical Helmholtz equation}

\author{Saumyajit Das}
\address{Harish-Chandra Research Institute, Homi Bhabha National Institute, Chhatnag Road, Jhunsi, Prayagraj (Allahabad) 211 019, India}
\email{saumyajit.math.das@gmail.com}

\author{Susovan Pramanik}
\address{Harish-Chandra Research Institute, Homi Bhabha National Institute, Chhatnag Road, Jhunsi, Prayagraj (Allahabad) 211 019, India}
\email{susovanpramanik@hri.res.in}

\begin{document}

	\begin{abstract}
		In this article we study the spectral problem related to the fractional Laplace-Beltrami equation on
		$(\mathbb{R}^d,g)$ and establish its equivalence with
		the classical anisotropic Helmholtz equation.  The proof is based on
		Seeley's construction of complex powers of elliptic operators and the
		pseudodifferential symbolic calculus.  As an application, we describe the
		related fixed-frequency inverse scattering problem of recovering the metric
		from the scattering amplitude.
	\end{abstract}
	\maketitle
	
	\section{Introduction}
	
	Fractional powers of elliptic operators arise naturally in anomalous diffusion,
	nonlocal wave propagation, and scattering theory.  In the Euclidean setting,
	the fractional Laplacian is the Fourier multiplier with symbol $|\xi|^{2s}$.
	For a variable Riemannian metric $g$, the corresponding object is the
	fractional Laplace--Beltrami operator $(-\Delta_g)^s$, which is most naturally
	understood through spectral theory and the pseudodifferential calculus.

	In this article we study the eigenvalue problem corresponding to the fractional
	Laplace--Beltrami operator in the setting of a Riemannian manifold
	$(\mathbb{R}^d,g)$, with smooth metric $g$. We assume throughout that $g$ is
	uniformly comparable with the Euclidean metric, namely
	\begin{align}\label{uniform ellipticiy}
		\theta |\xi|^2 \leq \sum_{i,j=1}^d g^{ij}(x)\xi_i\xi_j\leq \Theta |\xi|^2,
		\qquad x,\xi\in\mathbb{R}^d,
	\end{align}
	for some constants $0<\theta\leq \Theta<\infty$. We also assume that the
	coefficients of $g$ and their derivatives are bounded. This ensures that
	$P_g:=-\Delta_g$ is a positive uniformly elliptic differential operator of
	order two and that the standard pseudodifferential calculus applies.
	
	Authors in the articles \cite{cheng2023equivalence,guan2023helmholtz} studied
	related equivalence results in the Euclidean metric
	$(g_{ij}=\delta_{ij})$ setting and established a connection between the
	eigenfunctions of the fractional Schr\"odinger operator and those of the
	classical Schr\"odinger operator. The precise result is the following.
	\begin{thm}[{\cite[Theorem 1.1]{cheng2023equivalence}}]\label{comparability of fractional and local Helmholtz eq, eucleadian metric}
		Let $s>0$. Assume $u\in L^{\infty}(\R^d)$ satisfies $-(-\Delta)^su+u=0$ in $\mathcal{S}'(\R^d)$. Then $u\in C^{\infty}(\R^d)\cap L^{\infty}(\R^d)$ and $\Delta u+u=0$.
	\end{thm}
	Such results were also carried out by Guan, Murugan, Wei
	\cite{guan2023helmholtz} with some additional decay assumption at infinity.
	Here we prove the corresponding statement for the Laplace-Beltrami operator. Let
	\[
	-\Delta_g = -\frac{1}{\sqrt{|g|}} \sum_{i,j=1}^d \frac{\partial}{\partial x_i}
	\left( \sqrt{|g|} \, g^{ij} \frac{\partial}{\partial x_j} \right),
	\]
	where $g^{ij}=(g_{ij})^{-1}$ and $|g|=\det(g_{ij})$. Consider the fractional
	Helmholtz equation
	\begin{align}\label{fractional Helmholtz eq}
		-(-\Delta_g)^s u+k^{2s}u=0
		\qquad \text{in }\R^d, \qquad 0<s<1,
	\end{align}
	where $k>0$ is fixed. Equivalently,
	\[
	(P_g)^s u=k^{2s}u, \qquad P_g=-\Delta_g.
	\]
	
	When $g$ is Euclidean, the Herglotz wave function
	\[
	u_h(x):= \int_{\mathbb{S}^{d-1}} e^{ikx\cdot \theta} h(\theta) \, ds(\theta),
	\qquad h\in L^2(\mathbb{S}^{d-1}),
	\]
	solves the massless fractional Helmholtz equation. This follows from the
	Fourier multiplier identity
	\[
	(-\Delta)^s(e^{ix\cdot\xi})=|\xi|^{2s}e^{ix\cdot\xi}.
	\]
	For a variable metric, however, $(-\Delta_g)^s$ is no longer a Fourier
	multiplier. The correct replacement is Seeley's construction of complex powers
	of elliptic operators: since $P_g\in\Psi^2_{\mathrm{cl}}$ is positive elliptic,
	\[
	(P_g)^s\in\Psi^{2s}_{\mathrm{cl}},
	\qquad
	\sigma_{\mathrm{pr}}((P_g)^s)(x,\xi)=\big(g^{ij}(x)\xi_i\xi_j\big)^s.
	\]
	This observation allows us to factor the fractional equation by means of the
	pseudodifferential symbolic calculus.
	
	Let us state the main theorem of this article.
	\begin{thm}\label{equivalence result}
		Let $0<s<1$ and $k>0$. Let $g$ be a smooth uniformly elliptic metric on
		$\mathbb{R}^d$, with bounded derivatives as above. If
		$u\in L^{\infty}(\mathbb{R}^d)$ satisfies
		\[
		-(-\Delta_g)^s u+k^{2s}u=0
		\qquad \text{in } \mathcal{S}'(\R^d),
		\]
		then $u$ solves the classical anisotropic Helmholtz equation
		\[
		\Delta_g u+k^2u=0
		\qquad \text{in }\R^d,
		\]
		in the sense of distributions.
	\end{thm}


	For rational exponents $s=1/q$, one can formally obtain the conclusion by
	iterating $(-\Delta_g)^{1/q}$ and using the multiplicative property of spectral
	powers.  The pseudodifferential proof below treats all $0<s<1$ at once and does
	not require a separate rationality assumption on $s$.
	
	Let us assume the metric $g$ is Euclidean outside an Euclidean ball, i.e., $(g_{ij}(x)=\delta_{ij})$ outside the set $\{0\leq |x|< R\}$, for some $R>0$. Let us analyze the solution of the Helmholtz equation in $(\mathbb{R}^d,g)$ for such a given metric $g$. Given $k\in\R\setminus\{0\}$, $w\in\mathbb{S}^{d-1}$, the outgoing eigenfunctions, $u_g(\lambda,x,w)$ are solution of
	\begin{align}\label{local helmoholtz eq}
		\Delta_g \,u_g+k^2u_g=0 \ \ \ \mbox{in}\ \ \R^d, \ \ \ 0<s<1,
	\end{align}
	which have the asymptotic behavior instance in \cite[Lemma 19.3]{eskin2011lectures}:
	\begin{align}\label{scattering wave amp}
		u_g\sim e^{ikx\cdot w}+\frac{a_g(k,\theta,w)e^{ik|x|}}{|x|^{\frac{d-1}{2}}}+ \mathcal{O}\left( \frac{1}{|x|^{\frac{d+1}{2}}}\right),
	\end{align}
	where $\theta=\frac{x}{|x|}$.
	\medskip
	the function $a_g(k,\theta,w)$ is called the scattering amplitude. It measures, roughly speaking, the amplitude of the radial scattered wave which resulted from the interaction of the incident plane wave $e^{ikx\cdot w}$ with the perturbation of the Euclidean metric given by $g$. The inverse scattering problem is whether one can determine the metric $g$ from $a_g$ i.e., to study the non-linear map sending $g$ to $a_g$.
	\medskip
	We are interested in the ``fixed frequency" inverse scattering problem, that by measuring the scattering amplitude at a fixed frequency $k$ for all angles $w,\theta\in\mathbb{S}^{d-1}$. The scattering amplitude $a_g(k,\theta,w)$ depends on $2d-2$ variables. It is well known (see for instance \cite{uhlmann1992inverse}) that knowledge of $a_g(k,\theta,w)$ determines the set of Cauchy data
	\begin{align}\label{cauchy data}
		\mathcal{C}_{g,k}=\left\{ \left(u|_{\partial B}, \frac{\partial u}{\partial \nu}|_{\partial B}\right), \ \text{with}\ u\in H^2(B) \ \text{solution of}\ \ \Delta_g \,u+k^2u=0 \ \ \ \mbox{in}\ \ B.\right\}
	\end{align}
	for the Laplacian-Beltrami operator on $B$, where $B$ denotes the ball of radius $R$, centered at origin.
	\medskip
	Note that if $k$ is not a Dirichlet eigenvalue for the Laplace-Beltrami operator then the set of Cauchy data is the graph of the Dirichlet to Neumann map $\Lambda_{g,k}$. For more discussion on the anisotropic inverse scattering problem, we refer to \cite{uhlmann1992inverse} and \cite{isozaki2014recent}.
	
	As an application of Theorem \ref{equivalence result}, we discuss the related inverse problem of recovering the metric from the Cauchy data set (see \eqref{cauchy data}) for the fractional Helmholtz equation, and the inverse problem can be solved in a manner analogous to the local case. This analysis is presented in the final section of the article.
	\begin{rem}
		We note that this inverse scattering problem falls under the fixed-frequency case (with $k$ fixed). To the best of our knowledge, this is the first work addressing such a problem in the context of the fractional Laplace--Beltrami operator. For the variable-frequency case, we refer the readers to \cite{das2025inverse, das2026nonlocalnonlinearinversescattering}, where the authors recover the potential $V$ or the nonlinearity $f$ from the following equation
		\[
		(-\Delta)^s u-k^{2s}u -Vu=f(x,u), \ \text{in}\ \mathbb{R}^d,
		\]
		in Euclidean setting. 
	\end{rem}
	Theorem \ref{equivalence result} can be generalized to the relativistic Schr\"odinger operator $(-\Delta_g+\mu)$, where $\mu>0$ and $s\in(0,1)$, by employing essentially the same arguments. We note that the relativistic Schr\"odinger operator arises in the modeling of a variety of physical phenomena, including phase transitions, flame propagation, chemical reactions in liquids, and the Hamiltonian governing the motion of a free relativistic particle. This operator has attracted considerable attention in the mathematical literature, particularly in the study of Carleman estimates \cite{roncal2023carleman}, scattering theory \cite{ishida2020propagation}, and the analysis of relativistic operators \cite{lammerzahl1993pseudodifferential, ambrosio2022fractional, ambrosio2022note}. For additional background and developments, we refer the reader to \cite{byczkowski2009bessel, carmona1990relativistic, grzywny2008two}.
	
	There is also an extensive body of work devoted to the qualitative and analytical properties of solutions to the model equation
	\[
	(-\Delta + \mu)^s u = h(x,u),
	\]
	where $h(\cdot,\cdot)$ is a sufficiently regular nonlinearity arising from various physical applications. We refer the reader to \cite{ambrosio2022fractional, ambrosio2024concentration, byczkowski2009bessel, zelati2011existence, felmer2015scalar, ikoma2017existence} and the references therein for further details. We now present the analogue of the eigenfunction equivalence result for the relativistic Schr\"odinger equation.
	\begin{thm}\label{equivalence result, relativistic}
		Let $0<s<1$, $\mu>0$ and $k>0$. Let $g$ be a smooth uniformly elliptic metric on
		$\mathbb{R}^d$, with bounded derivatives as above. If
		$u\in L^{\infty}(\mathbb{R}^d)$ satisfies
		\[
		-(-\Delta_g+\mu)^s u+(k^{2s}+\mu)u=0
		\qquad \text{in } \mathcal{S}'(\R^d),
		\]
		then $u$ solves the classical anisotropic Helmholtz equation
		\[
		\Delta_g u+k^2u=0
		\qquad \text{in }\R^d
		\]
		in the sense of distributions.
	\end{thm}
	\begin{rem}
		Note that the parameter $\mu$ does not play any role in the local Helmholtz equation considered in the above theorem.
	\end{rem}

	\subsection{Fractional powers of the Laplace--Beltrami operator}
	
	There are several equivalent ways to define fractional powers of positive self-adjoint operators.
	In the present anisotropic setting, the most natural framework is the spectral functional calculus
	together with Seeley's construction of complex powers of elliptic operators \cite{Seeley1967}.
	Let
	\[
	P_g=-\Delta_g.
	\]
	Since \(P_g\) is a positive uniformly elliptic differential operator of order two,
	\[
	P_g\in \Psi^2_{\mathrm{cl}}(\mathbb R^d).
	\]
	Seeley's theorem implies that, for every \(z\in\mathbb C\),
	\[
	P_g^z\in \Psi^{2z}_{\mathrm{cl}}(\mathbb R^d).
	\]
	In particular, for \(0<s<1\),
	\[
	(-\Delta_g)^s=P_g^s\in \Psi^{2s}_{\mathrm{cl}}(\mathbb R^d),
	\]
	with principal symbol
	\[
	\sigma_{\mathrm{pr}}(P_g^s)(x,\xi)
	=
	\bigl(g^{ij}(x)\xi_i\xi_j\bigr)^s.
	\]
	
	The construction is based on the contour-integral representation
	\[
	P_g^{-z}
	=
	\frac{1}{2\pi i}
	\int_\Gamma
	\lambda^{-z}(P_g-\lambda)^{-1}\,d\lambda,
	\qquad \Re z>0,
	\]
	where \(\Gamma\) surrounds the spectrum of \(P_g\). The parameter-dependent symbolic
	expansion of the resolvent \((P_g-\lambda)^{-1}\) implies that \(P_g^z\) is again a classical
	pseudodifferential operator, whose principal symbol is the complex power of the principal
	symbol of \(P_g\).
	
	For comparison, in the Euclidean case \(g^{ij}=\delta_{ij}\), one recovers the Fourier multiplier identity
	\[
	\widehat{(-\Delta)^s f}(\xi)=|\xi|^{2s}\widehat f(\xi).
	\]
	For a variable metric, however, the symbol depends on \(x\), and therefore the argument
	cannot rely solely on the Fourier transform. The Fourier multiplier calculus is replaced by
	the pseudodifferential symbolic calculus for elliptic operators.
	
	The scalar function
	\[
	b_k(\lambda)=\frac{\lambda-k^2}{\lambda^s-k^{2s}}
	\]
	extends smoothly through \(\lambda=k^2\) (for details see Appendix, Lemma \ref{smoothness of bk}), since
	\[
	\lim_{\lambda\to k^2}\frac{\lambda-k^2}{\lambda^s-k^{2s}}
	=
	\frac{1}{s k^{2s-2}}.
	\]
	Hence
	\[
	B_k:=b_k(P_g)
	\]
	is a classical pseudodifferential operator of order \(2-2s\) (for details see Appendix, Lemma \ref{psedodifferential operator of degree}). Its principal symbol is
	\[
	\sigma_{\mathrm{pr}}(B_k)(x,\xi)
	=
	\frac{\sum_{i,j=1}^d g^{ij}(x)\xi_i\xi_j-k^2}
	{\bigl( \sum_{i,j=1}^dg^{ij}(x)\xi_i\xi_j\bigr)^s-k^{2s}},
	\]
	with the same smooth extension at the characteristic set
	\(\displaystyle{\sum_{i,j=1}^d g^{ij}(x)\xi_i\xi_j=k^2}\). By functional calculus,
	\begin{equation}\label{factorization_identity}
		P_g-k^2=(P_g^s-k^{2s})B_k.
	\end{equation}
	This factorization replaces the elementary Fourier multiplier factorization available in the Euclidean setting.

	\section{Proof of the theorem~\ref{equivalence result} and Theorem \ref{equivalence result, relativistic}}
	
	We prove only Theorem~\ref{equivalence result}. The proof of Theorem~\ref{equivalence result, relativistic} follows by a similar argument. We prove Theorem~\ref{equivalence result} using the factorization~\eqref{factorization_identity}. Let $u$ solve
	\[
	(P_g)^s u-k^{2s}u=0
	\]
	in the sense of distributions. Let $\phi\in\mathcal{S}(\mathbb{R}^d)$. Since
	$B_k\in\Psi^{2-2s}_{\mathrm{cl}}$, we have \cite{Hormander1985, Taylor1991}
	\[
	B_k\phi\in\mathcal{S}(\mathbb{R}^d).
	\]
	Therefore the fractional equation gives
	\[
	\big\langle (P_g)^s u-k^{2s}u, B_k\phi\big\rangle=0.
	\]
	Using the exact functional-calculus identity
	\[
	P_g-k^2=(P_g^s-k^{2s})B_k,
	\]
	we obtain
	\[
	\big\langle u,(P_g-k^2)\phi\big\rangle=0.
	\]
	Since $\phi\in\mathcal{S}(\mathbb{R}^d)$ was arbitrary, this proves
	\[
	(P_g-k^2)u=0
	\]
	in the sense of distributions. Recalling that $P_g=-\Delta_g$, we get
	\[
	(-\Delta_g-k^2)u=0,
	\]
	or equivalently
	\[
	\Delta_g u+k^2u=0.
	\]
	This completes the proof of Theorem~\ref{equivalence result}.

	\section{ Scattering Problem}
	Since $-\Delta_g-k^2$ is an elliptic pseudodifferential operator of order $2$, it is hypoelliptic \cite{Hormander1985, Taylor1991}. Hence, any distributional solution $u$ of
	$(-\Delta_g-k^2)u=0$ is smooth.
	Thanks to the equivalence result in Theorem~\ref{equivalence result}, the inverse scattering problem for the fractional Laplace--Beltrami Helmholtz equation 
	\[
	-(-\Delta_g)^s u+k^{2s}u=0, \ \ \text{in}\ \R^d, \ \ 0<s<1,\quad k>0 \ \text{is fixed},
	\]
	can be formulated in terms of inverse scattering problem for classical Helmholtz equation
	\[
	\Delta_g u +k^2u =0, \ \ \text{in}\ \R^d, \ \ k>0 \ \text{is fixed}.
	\]
	Let us assume the metric $g$ is Euclidean outside an Euclidean ball, i.e., $(g_{ij}(x)=\delta_{ij})$ outside the set $\{0\leq |x|< R\}$, for some $R>0$. Recall that  \eqref{scattering wave amp},
	\[
	u\sim e^{ikx\cdot w}+ \frac{a_g(k,\theta,w)e^{ik|x|}}{|x|^{\frac{d-1}{2}}}+\mathcal{O}\left(\frac{1}{|x|^{\frac{d+1}{2}}}\right), \ \ \text{where} \ \theta=\frac{x}{|x|},
	\]
	\subsection{Fixed Frequency Inverse Scattering Problem}
	The inverse problem states whether the metric $g$ can be recovered from the scattering amplitude function $a_g(k,\theta,w)$. This is known that the knowledge of $a_g(k,\theta,w)$ determines the Cauchy data set \cite{uhlmann1992inverse}
	\[
	\mathcal{C}_{g,k}=\left\{ \left(u|_{\partial B}, \frac{\partial u}{\partial \nu}|_{\partial B}\right), \ \text{with}\ u\in H^2(B) \ \text{solution of}\ \ \Delta_g \,u+k^2u=0 \ \ \ \mbox{in}\ \ B,\right\}
	\]
	where $B$ denotes the ball of radius $R$, centered at origin. Therefore, the question is now whether $\mathcal{C}_{g,k}$ determines $g$ or not. In the next section, we briefly discuss this using results from various previous works.
	
	
	It is easy to see that it is not possible to determine the metric uniquely from this information. Let $\Psi$ be a smooth diffeomorphism  of $\R^d$ which is the identity outside $B$. We define $v=u\circ \Psi^{-1}$. A straightforward calculation shows that $v$ satisfies 
	\begin{align}\label{conformal equation}
		\Delta_{\Psi*g}v+k^2 v=0 \ \ \text{in} \ \R^d, \ k>0 \ \text{fixed},
	\end{align}
	where $\Psi*g$ denotes the pull back of the metric $g$ under the diffeomorphism $\Psi$, that is
	\[
	\Psi*g= (D\Psi\circ g\circ D\Psi^T)\circ \Psi^{-1},
	\]
	and the Cauchy data
	\[
	\mathcal{C}_{\Psi*g,k}=\left\{ \left(v|_{\partial B}, \frac{\partial v}{\partial \nu}|_{\partial B}\right), \ \text{with}\ u\in H^2(B) \ \text{solution of}\ \ \Delta_{\Psi\ast g} \,v+k^2v=0 \ \ \ \mbox{in}\ \ B,\right\}
	\]
	becomes same with $\mathcal{C}_{g,k}$, i.e.,
	\begin{align}\label{conformal equivalence of Cauchy data set}
		\mathcal{C}_{\Psi*g,k}= \mathcal{C}_{g,k}.
	\end{align}
	The natural conjecture is that \eqref{conformal equivalence of Cauchy data set} is the only obstruction to uniqueness. We will comment on that shortly. 
	
	\medskip
	
	In general, the problem remains open for smooth metrics in dimensions \(d \geq 3\). In dimension \(d = 2\), the problem was addressed in \cite{sylvester1991inverse}, where the authors showed that the Cauchy data \(\mathcal{C}_{g,k}\) uniquely determines the metric \(g\) within the class of metrics conformal to the Euclidean metric. For dimensions \(d \geq 3\), only partial results are known. We refer the reader to \cite{lee1989determining}, where it was shown that \(\mathcal{C}_{g,k}\) uniquely determines \(g\) in the class of real-analytic metrics. The linearization of this problem was studied in \cite{sylvester1987global}.

	The two dimensional case is relatively easier thanks to the reduction of the problem to the isotropic case through the isothermal coordinate transformation \cite{ahlfors1966quasiconformal}. Using the change of variable, the Laplace-Beltrami operator  can be transformed to a conformal multiple of the standerd Laplacian (with respect to Euclidean metric). Thus we can transform \eqref{local helmoholtz eq} into
	\[
	c^2(x)\Delta+k^2
	\]
	with $c$ positive and equal to $1$ outside $B$. However, the recovery of the function $c$ from the scattering amplitude at a non zero fixed energy is still unsolved for general smooth function $c$. There are few known result under certain assumptions on $c$. We would like to refer the readers to \cite{sylvester1986uniqueness, sun1991generic, sun1990generic}, where the authors recover $c$ under a priori assumption that it is small or for some generic set of $c$'s.  The anisotropic conductivity equation, which is the analogue of \eqref{local helmoholtz eq} is given by
	\[
	\sum\limits_{i,j=1}^2 \frac{\partial}{\partial x_i} \gamma_{i,j}\frac{\partial u_{\gamma}}{\partial x_j}+k^2 u_{\gamma}=0
	\]
	with $\gamma=(\gamma^{i,j})$ a positive definite, symmetric smooth matrix which is identity outside $B$. The question remains same, at a fixed energy level, whether the Cauchy data $\mathcal{C}_{g,k}$ determines $\gamma$ uniquely up to conjugation by a group of diffeomorphism which is identity on the boundary of $B$ . Thanks to the work in \cite{ahlfors1966quasiconformal}, we again can transform this problem into an isotropic one  \cite{stein1970singular}. At zero energy level i.e., when $k=0$ the isotropic problem was solved in \cite{nachman1996global, brown1997uniqueness}, where in the later article, the problem was addressed with less regular conductivities. Uniqueness of such conductivities has not been proved yet at positive enegy level. Uniqueness has been proven only for small enough conductivities or for generic conductivities \cite{sun1991generic, sun1990generic}.

	\section{Appendix}
	We analyze the behaviour of the function $\displaystyle{b_k(\lambda)=\frac{\lambda-k^2}{\lambda^s-k^{2s}}}$. Recall that $k\neq 0$ and $s\in(0,1)$. We start with the following lemma.
	\begin{lem}\label{smoothness of bk}
		The function $\displaystyle{b_k(\lambda)=\frac{\lambda-k^2}{\lambda^s-k^{2s}}}$ is smooth for all $\lambda\geq 0$.
	\end{lem}
	\begin{proof}
		It is evident that the function is smooth on $[0,\infty)\setminus{k^2}$. Using L'Hospital's rule, we can extend the function continuously to $\lambda = k^2$ by defining
		\[
		b_k(\lambda)=\lim_{\lambda\to k^2}\frac{\lambda-k^2}{\lambda^s-k^{2s}}
		=\frac{1}{sk^{2(s-1)}}.
		\]
		We now determine the behavior of $b_k(\lambda)$ when $|\lambda-k^2|\ll 1$. Let $\lambda=k^2+h$, then
		\[
		b_k(\lambda)=\frac{h}{(k^2+h)^s-k^{2s}}= \frac{h}{k^{2s}(1+\frac{h}{k^2})^s-k^{2s}}.
		\]
		Since $\displaystyle{\left|\frac{h}{k^2}\right|\ll 1}$, the above expression can be written as
		\[
		b_k(\lambda)=\frac{h}{k^{2s}\sum\limits_{n=1}^{\infty} \frac{s(s-1)\cdots(s-n+1)}{n!}\left(\frac{h}{k^2}\right)^n}= \frac{1}{k^{2s}\sum\limits_{n=1}^{\infty} \frac{s(s-1)\cdots(s-n+1)}{n!}\frac{h^{n-1}}{k^{2n}}}.
		\]
		The series $\displaystyle{k^{2s}\sum\limits_{n=1}^{\infty} \frac{s(s-1)\cdots(s-n+1)}{n!}\frac{h^{n-1}}{k^{2n}}}$ is analytic when $\displaystyle{\left| \frac{h}{k^2}\right|\ll 1}$ and non zero when $h=0$. Hence $b_k$ is in fact analytic around $h=0$ i.e., around $\lambda=k^2$. Hence the function $b_k(\lambda)$ is smooth for all $\lambda\geq 0$.
	\end{proof}
	\begin{rem}\label{analyticity of bk}
		Since \(b_k(\lambda)\) admits an analytic extension at \(\lambda = k^2\), and for \(\lambda \neq k^2\) it is the quotient of two analytic functions whose denominator does not vanish, we can conclude that \(b_k(\lambda)\) is analytic on $(0,\infty)$.
	\end{rem}
	The next result addresses the asymptotic behaviours of derivatives of $b_k(\lambda)$ as $\lambda\to \infty$.
	\begin{lem}\label{decay of derivatives}
		Let $l\in\mathbb{N}\cup\{0\}$, then $\displaystyle{\left|D^l(b_k)(\lambda)\right|\lesssim |\lambda|^{1-s-l}}$, as $\lambda\to\infty$.
	\end{lem}
	\begin{proof}
		The case $l=0$ is standard. Let us choose $l\geq 1$. Using Leibnitz rule, we have that
		\begin{align}
			D^l (b_k)(\lambda)=& \sum_{j=0}^l \begin{pmatrix}
				l \\j   
			\end{pmatrix} D^j (\lambda-k^2) D^{l-j} \left(\frac{1}{\lambda^s-k^{2s}}\right) \nonumber\\
			=& (\lambda-k^2) D^{l} \left(\frac{1}{\lambda^s-k^{2s}}\right)+ l D^{l-1} \left(\frac{1}{\lambda^s-k^{2s}}\right). \label{derivative structure, bk}
		\end{align}
		Let us find out $\displaystyle{D^{p} \left(\frac{1}{\lambda^s-k^{2s}}\right)}$, for some $p\in\mathbb{N}\cup\{0\}$. Our ansatz is 
		\begin{equation}\label{Ansatz, derivative}
			\left \{
			\begin{aligned}
				D^{p} & \left(\frac{1}{\lambda^s-k^{2s}}\right)= \frac{P_{ls-l}(\lambda)}{(\lambda^s-k^{2s})^{p+1}}, \  \text{where}\ P_{ps-p}(\lambda):=\sum_{\alpha\in \mathbb{I}_{\alpha}} a_{\alpha}\lambda^{\alpha}, \\
				& \alpha, a_{\alpha}\in \mathbb{R}, \mathbb{I}_{\alpha} \ \text{is a finite index set},\ \max_{\alpha}\left\{\mathbb{I}_{\alpha}\right\}\leq ps-p.
			\end{aligned}
			\right. 
		\end{equation}
		It is clearly true for $p=0$, we do it via induction. Let the above anastz is true for $p$. The following calculation holds
		\begin{align*}
			D^{p+1} \left(\frac{1}{\lambda^s-k^{2s}}\right)= \frac{P_{ps-p}'(\lambda)(\lambda^s-k^{2s})-s(p+1)\lambda^{s-1}P_{ps-p}(\lambda)}{(\lambda^s-k^{2s})^{p+1}}.
		\end{align*}
		Note that $P_{ps-p}'(\lambda)(\lambda^s-k^{2s})-s(p+1)\lambda^{s-1}P_{ps-p}(\lambda)$ can be written as $\displaystyle{\sum_{\alpha\in \mathbb{I}_{\alpha}} a_{\alpha}\lambda^{\alpha}}$, where $\alpha,a_{\alpha}\in\mathbb{R}$, and  $\mathbb{I}_{\alpha}$ is a finite index set with $\displaystyle{\max_{\alpha}\left\{\mathbb{I}_{\alpha}\right\}\leq (p+1)s-(p+1)}$. Hence the ansatz \eqref{Ansatz, derivative} holds true for all $p\in\mathbb{N}\cup\{0\}$. Let's look into the relation \eqref{derivative structure, bk}. The $l$'th derivative  of the function $b_k(\lambda)$ can be expressed in the following way
		\[
		D^l(b_k)(\lambda)=\frac{(\lambda-k^2)P_{ls-l}(\lambda)+l(\lambda^s-k^{2s})P_{(l-1)s-(l-1)}(\lambda)}{(\lambda^s-k^{2s})^{l+1}}.
		\]
		Note that $\displaystyle{(\lambda-k^2)P_{ls-l}(\lambda)+l(\lambda^s-k^{2s})P_{(l-1)s-(l-1)}(\lambda)}$ can be written as $\displaystyle{\sum_{\alpha\in \mathbb{I}_{\alpha}} a_{\alpha}\lambda^{\alpha}}$, where $\alpha,a_{\alpha}\in\mathbb{R}$, and  $\mathbb{I}_{\alpha}$ is a finite index set with $\displaystyle{\max_{\alpha}\left\{\mathbb{I}_{\alpha}\right\}\leq ls-l+1}$. Hence
		\[
		\left | D^l(b_k)(\lambda)\right| \lesssim | \lambda|^{ls-s+1-ls-s}=|\lambda|^{1-s-l}, \quad \mbox{as}\ \lambda \to \infty.
		\]
	\end{proof}
	One can define the operator $B_k=b_k(-\Delta_g)$ through  functional calculus with the principle symbol 
	\[
	\sigma_{\mathrm{pr}}(B_k)(x,\xi)
	=
	\frac{\sum_{i,j=1}^d g^{ij}(x)\xi_i\xi_j-k^2}
	{\bigl(\sum_{i,j=1}^d g^{ij}(x)\xi_i\xi_j\bigr)^s-k^{2s}},
	\]
	where $x,\xi\in \mathbb{R}^d$. We have the following lemma.
	\begin{lem}\label{towards functional calculus}
		Let $\bigl(g^{ij}(x)\bigr)_{i,j=1}^{d}$ be a uniformly elliptic, smooth, and bounded matrix-valued function whose derivatives of all orders are bounded. Then, for every pair of multi-indices $\alpha,\beta\in (\mathbb{N}\cup{0})^{d}$, the estimate
		\[
		\left|\partial_x^{\alpha}\partial_{\xi}^{\beta} b_k\left(\sum_{i,j=1}^d g^{ij}(x)\xi_i\xi_j\right)\right|
		\lesssim
		\langle \xi\rangle^{2-2s-|\beta|}
		\]
		holds uniformly for all $x,\xi\in\mathbb{R}^{d}$. Here,
		\[
		\langle \xi\rangle
		:=
		\left(1+\sum_{i=1}^{d}|\xi_i|^2\right)^{1/2}
		\]
		denotes the Japanese bracket. Moreover, for a multi-index
		$r=(r_1,\ldots,r_d)\in(\mathbb{N}\cup{0})^d$,
		we define
		\[
		|r|:=\sum_{i=1}^{d} r_i.
		\]
	\end{lem}
	\begin{proof}
		Thanks to uniform ellipticity of the matrix valued function $g^{ij}(x)$, we have that
		\[
		\theta |\xi|^2 \leq \sum_{i,j=1}^d g^{ij}(x)\xi_i\xi_j\leq \Theta |\xi|^2,
		\qquad x,\xi\in\mathbb{R}^d,
		\]
		where $\theta, \Theta$ are some positive constants. Let $q(x,\xi):=\sum_{i,j=1}^d g^{ij}\xi_i\xi_j$. We use Fa\'a di Bruno's formula to compute the mixed derivatives. We have that  
		\[
		\partial^{\beta}_{\xi} b_k(q)= \sum_{m=1}^{|\beta|} (D^m b_k)(q) \prod_{i=1}^m \partial_{\xi}^{\mu_i} q, 
		\]
		where $\mu_i\in (\mathbb{N}\cup\{0\})^d$ for all $i=1,\cdots,m$ and $\displaystyle{\sum_{i=1}^m |\mu_i|=|\beta|}$. Since $q(x,\xi)$ is a polynomial of order two, $\partial_{\xi}^{\mu_i}q$ vanishes where $|\mu_i|\geq 3$. Hence $\displaystyle{\prod_{i=1}^m \partial_{\xi}^{\mu_i} q}$ is a polynomial of degree $\displaystyle{\sum_{i=1}^m (2-|\mu_i|)}=2m-|\beta|$ with coefficients that are smooth bounded functions of $x$. Moreover, all derivatives of these coefficients are uniformly bounded. Next we use Leibnitz rule. 
		\[
		\partial_{x}^{\alpha}\partial_{\xi}^{\beta} b_k(q)= \sum_{m=1}^{|\beta|} \left( \sum_{\nu\leq \alpha} \begin{pmatrix}
			\alpha\\ \nu
		\end{pmatrix} \partial^{\nu}_{x}(D^m b_k)(q) \partial^{\alpha-\nu}_{x}\left(\prod_{i=1}^m \partial_{\xi}^{\mu_i} q\right)\right),
		\]
		where $\nu\in(\mathbb{N}\cup\{0\})^d$, $\nu\leq \alpha$ means $\nu\leq \alpha_i$ for all $i=1,\cdots, d$ and $\displaystyle{\begin{pmatrix}
				\alpha\\ \nu
			\end{pmatrix}=\prod_{i=1}^d \begin{pmatrix}
				\alpha_i\\ \nu_i
		\end{pmatrix}}$. We use Fa\`a di Bruno's formula once again. It yields
		\begin{align}\label{intermediate 1}
			\partial_{x}^{\alpha}\partial_{\xi}^{\beta} b_k(q)= \sum_{m=1}^{|\beta|} \left( \sum_{\nu\leq \alpha} \begin{pmatrix}
				\alpha\\ \nu
			\end{pmatrix} \left(\sum_{j=1}^{|\nu|} (D^{m+j} b_k)(q) \prod_{l=1}^{j} \partial_x^{\delta_l} q(x)\right) \partial^{\alpha-\nu}_{x}\left(\prod_{i=1}^m \partial_{\xi}^{\mu_i} q\right)\right),
		\end{align}
		where $\delta_l\in (\mathbb{N}\cup\{0\})^d$ for all $l=1,\cdots,j$ and $\displaystyle{\sum_{l=1}^j |\delta_l|=|\nu|}$. Note that the function $q$ is positive thanks to uniform ellipticity. We can further say that the term $\displaystyle{\partial^{\alpha-\nu}_{x}\left(\prod_{i=1}^m \partial_{\xi}^{\mu_i} q\right)}$ is a polynomial (in $\xi$) of degree $2m-|\beta|$ with smooth and bounded coefficients in `$x$' and similarly the term $\displaystyle{\prod_{l=1}^{j} \partial_x^{\delta_l} q(x)}$ is a polynomial (in $\xi$) of degree $\displaystyle{\sum_{l=1}^j 2=2j}$ with smooth bounded coefficients in `$x$'. Moreover, thanks to Lemma \ref{decay of derivatives}, we have that 
		\[
		\left| (D^{m+j} b_k)(q)\right| \lesssim | q|^{1-s-m-j}, \quad \mbox{as} \ q\to \infty. 
		\]
		Thanks to uniform ellipticity again we can write the above expression as 
		\[
		\left| (D^{m+j} b_k)(q)\right| \lesssim | \xi|^{2-2s-2m-2j}, \quad \mbox{as} \ \xi\to \infty.
		\]
		Hence, from \eqref{intermediate 1}, we conclude that
		\[
		\left|\partial_x^{\alpha}\partial_{\xi}^{\beta} b_k\left(\sum_{i,j=1}^d g^{ij}(x)\xi_i\xi_j\right)\right|
		\lesssim
		\langle \xi\rangle^{2-2s-2m-2j+2j+2m-|\beta|}=  \langle \xi\rangle^{2-2s-|\beta|}, \quad \forall\, x, \xi\in\mathbb{R}^d.
		\]
	\end{proof}
	As an immediate consequence  of the above lemma, we have the following result.
	\begin{lem}\label{psedodifferential operator of degree}
		The operator $B_k := b_k(-\Delta_g)$ is an elliptic pseudodifferential operator of order $2-2s$ and the principle symbol is given by 
		\[
		\sigma_{\mathrm{pr}}(B_k)(x,\xi)
		=
		\frac{\sum_{i,j=1}^d g^{ij}(x)\xi_i\xi_j-k^2}
		{\bigl( \sum_{i,j=1}^dg^{ij}(x)\xi_i\xi_j\bigr)^s-k^{2s}}.
		\]
	\end{lem}
	\begin{proof}
		We only show the ellipticity. Rest follows from Lemma \ref{towards functional calculus}. Let $\lambda=\sum_{i,j=1}^dg^{ij}(x)\xi_i\xi_j$. The principle symbol can be expressed as 
		\begin{align}\label{intermediate 2}
			\sigma_{\mathrm{pr}}(B_k)(x,\xi)= \lambda^{1-s}\frac{(1-k^2\lambda^{-1})}{(1-k^{2s}\lambda^{-s})}.
		\end{align}
		We will show that the function $f:(0,\infty)\to \mathbb{R}$, defined by $\displaystyle{f(\lambda):=\frac{(1-k^2\lambda^{-1})}{(1-k^{2s}\lambda^{-s})}}$, is positive and bounded from below when $\lambda$ significantly large. Clearly, when $\lambda>k^2$, both the numerator and denominator are positive and hence $f(\lambda)>0$. Similarly $f(\lambda)>0$, when $\lambda<k^2$ also. When $\lambda=k^2$, we have that $\displaystyle{f(\lambda=k^2)=\lim_{\lambda\to k^2}\frac{(1-k^2\lambda^{-1})}{(1-k^{2s}\lambda^{-s})}=\frac{1}{s}>0}$. Next consider $\lambda>2k^2$, then $\displaystyle{f(\lambda)\geq \frac{1}{1-2^{-s}}}$. Hence, thanks to uniform ellipticity \eqref{uniform ellipticiy}, e have that
		\[
		\sigma_{\mathrm{pr}}(B_k)(x,\xi) \gtrsim |\xi|^{2-2s}, \quad \forall \, x,\xi\in\mathbb{R}^d\, \ \text{with}\ |\xi|\geq \frac{2k^2}{\theta},
		\]
		where $\theta$ is as defined in \eqref{uniform ellipticiy}. Thus we conclude the proof.
	\end{proof}

	\vspace{.3cm}
	\textbf{Acknowledgment:} The authors were funded by the Department of Atomic Energy $($DAE$)$, Government of India. 
	
	The authors sincerely express gratitude to Prof. Tuhin Ghosh (Harish-Chandra Research Institute, Prayagraj, Uttar Pradesh, India) for his valuable insights and feedback, which significantly contributed to improve this article.


	\vspace{.3cm}
	\textbf{Data Availability:} The authors shall permit all the data underlying the findings of this manuscript to be shared by any researchers or groups who are interested in the article.

	\vspace{.3cm}
	\textbf{Declaration:}

	\vspace{.3cm}
	\textbf{Conflict of Interest:} The authors declare that there is no conflict of interest regarding the publication of this paper.
	

	{\small 
		\bibliographystyle{alpha}
		\bibliography{ref}}

\end{document}